\documentclass[12pt]{article}

\usepackage{geometry}               
\usepackage{color,amsmath,amsfonts,amssymb,amsthm}
\usepackage{epstopdf}
\usepackage{url}

\newtheorem{mythm}{Theorem}

\newtheorem{prop}[mythm]{Proposition}

\newtheorem*{claim}{Claim}

\usepackage[pdftex]{graphicx}

\usepackage{mathrsfs}

\title{The abelian complexity of the paperfolding word}

\author{Blake Madill and Narad Rampersad\\
Department of Mathematics and Statistics, University of Winnipeg\\
515 Portage Ave., Winnipeg, Manitoba, R3B 2E9 Canada\\
\texttt{\{blakemadill, narad.rampersad\}@gmail.com}}

\begin{document}
\maketitle
\begin{abstract}
We show that the abelian complexity function of the ordinary
paperfolding word is a $2$-regular sequence.
\end{abstract}

\section{Introduction}
In this paper we study the abelian complexity function of the ordinary
paperfolding word (here defined over $\{0,1\}$)
\[
{\bf f} = 0010011000110110001001110011011\cdots
\]

The \emph{(subword) complexity function} of an infinite word ${\bf w}$
is the function of $n$ that counts the number of distinct factors (or
blocks) of ${\bf w}$ of length $n$.  Allouche \cite{All92} determined
that the number of factors of length $n$ of the ordinary paperfolding
word (indeed of any paperfolding word) is $4n$ for $n \geq 7$.

The \emph{abelian complexity function} is the function of $n$ that
counts the number of abelian equivalence classes of factors of ${\bf
  w}$ of length $n$.  That is, we define an equivalence relation on
factors of ${\bf w}$ of the same length by saying that $u$ and $v$ are
\emph{abelian equivalent} if $u$ can be obtained by rearranging the
symbols of $v$.  The abelian complexity function counts the number of
such equivalence classes for each length $n$.  The abelian complexity
function $\rho(n)$ for the ordinary paperfolding word thus has the
following initial values:
\[
\begin{array}{|c|*{20}{c}|}
\hline
n&1&2&3&4&5&6&7&8&9&10&11&12&13&14&15&16&17&18&19&20\\
\hline
\rho(n)&2&3&4&3&4&5&4&3&4&5&6&5&4&5&4&3&4&5&6&5\\
\hline
\end{array}
\]

The study of the abelian complexity of infinite words is a relatively
recent notion and was first introduced by Richomme, Saari, and Zamboni
\cite{RSZ11}.  A series of subsequent papers have pursued the topic
\cite{BBT11, CR11, CRSZ11, RSZ10, Saa09, Tur10, Tur12}.  Some notable
sequences whose abelian complexity functions have been determined
include the Thue--Morse word \cite{RSZ11} and all Sturmian words (a
classical result \cite{CH73}).  Balkov\'a, B\v{r}inda, and Turek
\cite{BBT11} and Turek \cite{Tur10,Tur12} computed the abelian
complexity functions for some other classes of infinite words.
Richomme, Saari, and Zamboni \cite{RSZ10} determined the range of
values of the abelian complexity function of the Tribonacci word, but
they did not obtain a precise characterization of this function.  In
all of these cases, the words studied have a bounded abelian complexity
function.  However, the paperfolding word has an unbounded abelian
complexity function.  To the best of our knowledge the present paper
is the first to compute precisely the abelian complexity function of
an infinite word in the case where this function grows unboundedly
large.

We do not obtain a closed form for the abelian complexity function of
the paperfolding word; rather, we show that it is $2$-regular (see
\cite{AS03}), and provide a finite list of recurrence relations that
determine the function.  The ordinary paperfolding word is an example
of a $2$-automatic sequence (again see \cite{AS03}).  Recently,
Shallit and his co-authors have developed and exploited techniques for
algorithmically deciding many interesting properties of automatic
sequences \cite{ARS09,CRS11,GS12a,GS12b,GSS12,HS12,SS11}.  In
particular, Charlier, Rampersad, and Shallit \cite{CRS11} showed that
the subword complexity function of a $k$-automatic sequence is
$k$-regular and gave an algorithmic method to compute this function
(see also the recent improvement by Goc, Schaeffer, and Shallit
\cite{GSS12}).  However, this algorithmic methodology does not seem to
be applicable to any questions concerning ``abelian'' properties of
words.  For example, Holub \cite{Hol12} recently showed that the
paperfolding words contain arbitrarily large abelian powers; his
method was ad hoc, since the algorithmic techniques described above do
not seem to apply.  Our study of the abelian complexity function of
the paperfolding words is similarly ad hoc.

\section{Preliminaries}

We let
\[
{\bf f} = (f_n)_{n \geq 1} = 0010011000110110001001110011011\cdots
\]
denote the ordinary paperfolding word.  We have already made certain choices
in this statement, the first being that we are taking the
paperfolding word to be defined over the alphabet $\{0,1\}$, rather
than over $\{+1,-1\}$, which in some circumstances may be a more
natural choice.  We have also chosen to index the terms of the
paperfolding word starting with $1$, rather than $0$.

There are several ways to define the ordinary paperfolding word.  The
``number-theoretic'' definition is as follows.  For $n \geq 1$, write
$n = n'2^k$, where $n'$ is odd.  Then
\begin{equation}\label{odd_part}
f_n =
\begin{cases}
0 & \text{if } n' \equiv 1 \pmod 4 \\
1 & \text{if } n' \equiv 3 \pmod 4.
\end{cases}
\end{equation}
Another definition of the paperfolding word, of which we shall make
frequent use in the sequel, is that obtained by the so-called
\emph{Toeplitz construction}:
\begin{itemize}
\item
  Start with an infinite sequence of \emph{gaps}, denoted by ?.
  \[\begin{array}{*{16}{c}}
  ? & ? & ? & ? & ? & ? & ? & ? & ? & ? & ? & ? & ? & ? & ? & \cdots
  \end{array}\]
\item
  Fill every other gap with alternating $0$'s and $1$'s.
  \[\begin{array}{*{16}{c}}
  0 & ? & 1 & ? & 0 & ? & 1 & ? & 0 & ? & 1 & ? & 0 & ? & 1 & \cdots
  \end{array}\]
\item
  Repeat.
  \[\begin{array}{*{16}{c}}
  0 & 0 & 1 & ? & 0 & 1 & 1 & ? & 0 & 0 & 1 & ? & 0 & 1 & 1 & \cdots
  \end{array}\]
  \[\begin{array}{*{16}{c}}
  0 & 0 & 1 & 0 & 0 & 1 & 1 & ? & 0 & 0 & 1 & 1 & 0 & 1 & 1 & \cdots
  \end{array}\]
  \[\begin{array}{*{16}{c}}
  0 & 0 & 1 & 0 & 0 & 1 & 1 & 0 & 0 & 0 & 1 & 1 & 0 & 1 & 1 & \cdots
  \end{array}\]
\end{itemize}
In the limit, one obtains the ordinary paperfolding word.

Our main result concerns the abelian complexity function of ${\bf f}$.
Let us first define an equivalence relation $\sim$ on  words over
$\{0,1\}$ by $$u\sim v\mbox{ if }u\mbox{ is an anagram of }v.$$
If for $a \in \{0,1\}$ we write $|w|_a$ to denote the number of
occurrences of $a$ in the word $w$, then this definition amounts to
saying that $u \sim v$ if $|u|_a = |v|_a$ for all $a \in \{0,1\}$.
For example,  $00011 \sim 01010$.

If ${\bf w}$ is an infinite word, the \emph{abelian complexity
  function} of ${\bf w}$ is the function $\rho_{\bf w}^{ab} :
\mathbb{N} \to \mathbb{N}$, where for $n = 1,2,\ldots$, the value of
$\rho_{\bf w}^{ab}(n)$ is the number of distinct equivalence classes
of $\sim$ over all factors of length $n$ of ${\bf w}$.

Our goal is to show that $(\rho(n))_{n \geq 1} = (\rho_{\bf
  f}^{ab}(n))_{n \geq 1}$ is a
$2$-regular sequence.  To explain this concept we first define the
$k$-kernel of a sequence.  Let $k \geq 2$ be an integer and let ${\bf
  w} = (w(n))_{n \geq 0}$ be an infinite sequence of integers.  The
\emph{$k$-kernel} of ${\bf w}$ is the set of subsequences
\[
\mathcal{K}_k({\bf w}) = \{(w(k^en + c))_{n \geq 0} : e \geq 0, 0 \leq c < k^e\}.
\]
The sequence ${\bf w}$ is \emph{$k$-automatic} if its $k$-kernel is
finite.  For example, it is easy to verify that ${\bf f}$ is a
$2$-automatic sequence.  The sequence ${\bf w}$ is \emph{$k$-regular}
if the $\mathbb{Z}$-module generated by its $k$-kernel is finitely
generated;  that is, if there exists a finite subset $\{{\bf w}_1,
\cdots, {\bf w}_d\} \subseteq \mathcal{K}_k({\bf w})$ such that any
element of $\mathcal{K}_k({\bf w})$ can be written as a
$\mathbb{Z}$-linear combination of the ${\bf w}_i$, along with,
possibly, the constant sequence $(1)_{n \geq 0}$.  For further
details, see \cite{AS03}.

\section{$2$-regularity of $\rho(n)$}

From the previous discussion concerning $k$-regular sequences, it is
clear that

\begin{mythm}\label{2-regular}
The abelian complexity function $\rho(n) =\rho_{\bf f}^{ab}(n)$ of the
ordinary paperfolding word is $2$-regular.
\end{mythm}

\noindent is an immediate consequence of the more precise

\begin{mythm}\label{identities}
The function $\rho(n)$ satisfies the relations
\begin{eqnarray*}
\rho(4n) &=& \rho(2n) \\
\rho(4n+2) &=& \rho(2n+1)+1\\
\rho(16n+1) &=& \rho(8n+1)\\
\rho(16n+\{3,7,9,13\}) &=&
\rho(2n+1)+2\\
\rho(16n+5) &=& \rho(4n+1)+2\\
\rho(16n+11) &=& \rho(4n+3)+2\\
\rho(16n+15) &=& \rho(2n+2)+1.
\end{eqnarray*}
\end{mythm}

\noindent since the latter implies that the $\mathbb{Z}$-module
generated by the $2$-kernel of $(\rho(n))_{n \geq 0}$ is generated by the finite
set
\[
\{ (\rho(2n+1))_{n \geq 0}, (\rho(2n+2))_{n \geq 0}, (\rho(4n+1))_{n
  \geq 0}, (\rho(4n+3))_{n \geq 0}, (\rho(8n+1))_{n \geq 0}, (1)_{n
  \geq 0}\}.
\]

Before starting the proof, we introduce some notation.  Let $B_n$ be
the set of factors of ${\bf f}$ of length $n$.  We define functions
$\Delta:\{0,1\}^* \rightarrow \mathbb{Z}$ and $M:\mathbb{N}
\rightarrow \mathbb{Z}$ by $\Delta(w)=|w|_0-|w|_1$ and
$M(n)=\max\{\Delta(w):w\in B_n\}$.  A word $w \in B_n$ is
\emph{maximal} if $\Delta(w) = M(n)$.  Note that a word beginning and
ending with $1$ cannot be maximal.  For any word $w$ over $\{0,1\}$ we
write $\overline{w}$ for the complement of $w$; that is, the word
$\overline{w}$ is obtained from $w$ by changing $0$'s into $1$'s and
$1$'s into $0$'s.  We also denote the reversal of $w$ by $w^R$; that
is, if $w = w_1\cdots w_n$ then $w^R = w_n \cdots w_1$.  It is well
known that if $w$ is a factor of ${\bf f}$ then so is
$\overline{w}^R$.  Also we shall always use the notation such that
$x,y,z,u_i\in\{0,1\}$ for $i\in\mathbb{N}$.

We begin by establishing the following relationship between $\rho(n)$
and $M(n)$.

\begin{claim}
$\rho(n) = M(n) + 1.$
\end{claim}

\begin{proof}
It is clear that the abelian equivalence class of a word $w$ is
determined by its length $|w|=n$ and the value $\Delta(w)$.  Hence
$\rho(n) = |\Delta(B_n)|$.  Furthermore, since for any factor $w$ of
${\bf f}$, the reverse complement $\overline{w}^R$ also occurs in
${\bf f}$, and $\Delta(\overline{w}^R) = -\Delta(w)$, we see that
$\Delta(B_n)$ consists of values between $-M(n)$ and $M(n)$.  Let us
order the values of $\Delta(B_n)$:
\[
\Delta(B_n) = \{-M(n) = \Delta_1 < \Delta_2 < \cdots <
\Delta_{|\Delta(B_n)|} = M(n)\}.
\]
It is not hard to see that $\Delta_{i+1} - \Delta_i = 2$, so we
conclude that $|\Delta(B_n)| = M(n)+1$, as claimed.
\end{proof}

It follows that to prove any of the identities of
Theorem~\ref{identities}, it suffices to prove the corresponding
relation with $M$ in place of $\rho$.  For example, to show that
$\rho(16n+1) = \rho(8n+1)$, we may equivalently show that $M(16n+1) =
M(8n+1)$.  It is this approach that we shall take to prove
Theorem~\ref{identities}.

There is one more fact that we need to establish before proceeding
with the proof.

\begin{claim}
$\rho(n+1) = \rho(n) \pm 1$.
\end{claim}

\begin{proof}
Let $w \in B_n$ be a word satisfying $\Delta(w) = M(n)$.  If there
exists $w' \in B_{n+1}$ such that $w' \sim w0$, then $M(n+1) = M(n)
+ 1$ and so $\rho(n+1) = \rho(n)+1$.  If not, then there exists $w' \in
B_{n+1}$ such that $w' \sim w1$.  In this case $M(n+1) = M(n) - 1$ and
so $\rho(n+1) = \rho(n) - 1$.
\end{proof}

We now prove each of the relations of Theorem~\ref{identities}.

\begin{claim}
$\rho(4n)=\rho(2n)$.
\end{claim}

\begin{proof}
Let $w=w_1w_2\cdots w_{2n}\in B_{2n}$ such that $\Delta(w)=M(2n)$. Then we know that $w':=xw_1\overline{x}w_2\cdots\overline{x}w_{2n}\in B_{4n}$ and $\Delta(w')=\Delta(w)$. We claim that $\Delta(w')=M(4n)$. Suppose there was a factor $z\in B_{4n}$ such that $\Delta(z)>\Delta(w')$. Then 
$$
z=yz_1\overline{y}z_2\cdots \overline{y}z_{2n} \mbox{ or } z=z_1\overline{y}z_2\cdots \overline{y}z_{2n}y.
$$
Furthermore $z_1z_2\cdots z_{2n}\in B_{2n}$ and $\Delta(z_1z_2\cdots z_{2n})>\Delta(w)$,
which is a contradiction. Therefore $M(2n)=M(4n)$ and so
$\rho(2n)=\rho(4n)$.
\end{proof}

\begin{claim}
$\rho(4n+2)=\rho(2n+1)+1$.
\end{claim}

\begin{proof}
Let $w=w_1w_2\cdots w_{2n+1}\in B_{2n+1}$ such that $\Delta(w)=M(2n+1)$. Now let $w'\in B_{4n+2}$ such that $\Delta(w')=M(4n+2)$. Then we know that $w'$ is a factor of 
$$
xw_1\overline{x}w_2\cdots x w_{2n+1}\overline{x}.
$$
If $x=0$ then we may choose $w'=xw_1\overline{x}w_2\cdots x w_{2n+1}$ and if $x=1$ we may choose $w'=w_1\overline{x}w_2\cdots x w_{2n+1}\overline{x}$. In either case we have $M(4n+2)=M(2n+1)+1$. Therefore $\rho(4n)=\rho(2n+1)+1$, as desired.
\end{proof}

\begin{claim}
$\rho(16n+1)=\rho(8n+1)$.
\end{claim}

\begin{proof}
Let $w\in B_{16n+1}$ such that $\Delta(w)=M(16n+1)$. Then either
\begin{enumerate}
\item $w=w_1xw_2 \cdots \overline{x}w_{8n+1}$, or
\item $w=xw_1\overline{x}w_2\cdots \overline{x}w_{8n}x$.
\end{enumerate}
If $w$ is of type (i) then $\Delta(w)=\Delta(w_1w_2\cdots
w_{8n+1})\leq M(8n+1)$. If $w$ is of type (ii) then by the maximality of
$w$ we must have $x=0$ and so $\Delta(w)=\Delta(w_1w_2\cdots w_{8n})+1
\leq M(8n)+1$. We now show that $M(8n+1)=M(8n)+1$. 
To show this let $z\in B_{8n}$ such that $\Delta(z)=M(8n)$. Then $z$ is a factor of
$$
z':=y\overline{x}\overline{y}z_1yx\overline{y}z_2\cdots z_{2n}y\overline{x}\overline{y}\in B_{8n+3}
$$
where $\Delta(z_1z_2\cdots z_{2n})=M(2n)=M(8n)$. Then we may choose $z=z_1yx\overline{y}\cdots z_{2n}y\overline{x}\overline{y}$ so that $\Delta(z)=M(8n)$.
If $y=1$ then $\Delta(\overline{y}z_1yx\overline{y}z_2\cdots
z_{2n}y\overline{x}\overline{y})=M(8n)+1=M(8n+1)$. If $y=0$ then
$\Delta(y\overline{x}\overline{y}z_1\cdots z_{2n}y)=M(8n)+1=M(8n+1)$.
Therefore $M(8n+1)=M(8n)+1$.

To show that $M(16n+1) \geq M(8n+1)$, note that if $w_1w_2\cdots
w_{8n+1}$ is maximal, then $w=w_1xw_2 \cdots \overline{x}w_{8n+1}$ is
a factor of ${\bf f}$ and $\Delta(w)=M(8n+1)$.  Hence $M(16n+1) =
M(8n+1)$.
\end{proof}

\begin{claim}
$\rho(16n+3)=\rho(2n+1)$.
\end{claim}

\begin{proof}
Let $w\in B_{16n+3}$ such that $\Delta(w)=M(16n+3)$. Then we know that $w$ is a factor of
\begin{enumerate}
\item $xy\overline{x}zx\overline{yx}u_1\cdots\overline{z}x\overline{yx}u_{2n}xy\overline{x}zx\overline{y}$ or,
\item $xy\overline{x}u_1x\overline{yx}\overline{z}\cdots u_{2n}x\overline{yx}zxy\overline{x}u_{2n+1}x\overline{y}.$
\end{enumerate}
In either case, by the Toeplitz construction, the position of $u_1$ in ${\bf f}$ is
congruent to $0\pmod{8}$. Thus in case (i) the initial $xy$ starts at a position which
is congruent to $1\pmod{8}$ and so by \eqref{odd_part} we have
$x=y=0$. In case (ii), the initial $xy$ begins at a position congruent
to $5\pmod{8}$ and so by \eqref{odd_part} we have $x=0$ and $y=1$.\\
\underline{Case 1}. Suppose $w$ is a factor of (i). Then we have that
$$
\Delta(w)=
\begin{cases}
\Delta(u_1u_2\cdots u_{2n})+1 &\ \mbox{if $w$ begins in the first position}\\
\Delta(u_1u_2\cdots u_{2n})+\Delta(z) &\ \mbox{if $w$ begins in the second, third, or fourth position}.
\end{cases}
$$
\underline{Case 2}. Suppose $w$ is a factor of (ii). Then we have that
$$
\Delta(w)=
\begin{cases}
\Delta(u_1u_2\cdots u_{2n})-1 &\ \mbox{if $w$ begins in the first position}\\
\Delta(u_1u_2\cdots u_{2n+1})-2 &\ \mbox{if $w$ begins in the second position}\\
\Delta(u_1u_2\cdots u_{2n+1}) &\ \mbox{if $w$ begins in the third position}\\
\Delta(u_1u_2\cdots u_{2n+1})+2 &\ \mbox{if $w$ begins in the fourth position}.
\end{cases}
$$
In any event, we have that $M(16n+3)\leq M(2n+1)+2$ so that $\rho(16n+3)\leq\rho(2n+1)+2$.

Now let $\Delta(u_1u_2\cdots u_{2n+1})=M(2n+1)$. Let 
$$w=u_1x\overline{yxz}\cdots
u_{2n}x\overline{yx}zxy\overline{x}u_{2n+1}x\overline{y}\in
B_{16n+3},$$
where $xy=01$, so that $\Delta(w)=M(2n+1)+2$. Thus $\rho(16n+3)\geq\rho(2n+1)+2$ and so the result follows.
\end{proof}

\begin{claim}
$\rho(16n+5)=\rho(4n+1)+2$.
\end{claim}

\begin{proof}
Let $w\in B_{16n+5}$ such that $\Delta(w)=M(16n+5)$. Then this factor occurs (in the paperfolding word) at position 1, 2, 3, or 4 of a factor of the form
$$
xy\overline{x}w_1x\overline{yx}w_2\cdots xy\overline{x}w_{4n+1}x\overline{yx}w_{4n+2}.
$$
If $w$ starts at positions 1, 2, or 3, it is easily verified that
$$\Delta(w)\leq\Delta(w_1w_2\cdots w_{4n+1})+2.$$ Also if $w$ starts at
position 4, then $$\Delta(w)\leq \Delta(w_1w_2\cdots w_{4n+2}) + 1 \leq
\Delta(w_1w_2\cdots w_{4n+1}) + 2.$$  Thus it suffices to find a $w\in
B_{16n+5}$ such that $\Delta(w)=M(4n+1)+2$.

Let $v=w_1 w_2\cdots w_{4n+1}\in B_{4n+1}$ such that $\Delta(v)=M(4n+1)$. We have either
\begin{enumerate}
\item $v=zu_1\overline{z}u_2\cdots\overline{z}u_{2n}z$ or 
\item $v=u_1\overline{z}u_2\cdots\overline{z}u_{2n}zu_{2n+1}$,
\end{enumerate}
where $u_i,z\in\{0,1\}$. However, case~(ii) can always be reduced to
case~(i), since in case~(ii) the word $v$ is preceeded by $z$ and
followed by $\overline{z}$ in ${\bf f}$, and so we can always find a
factor $v'$ of form~(i) such that $\Delta(v') = \Delta(v)$. So let $v$
be of the form corresponding to case~(i) above. Then we know
$$
w':=xy\overline{x}zx\overline{yx}u_1xy\overline{xz}\cdots xy\overline{x}zx\in B_{16n+5}.
$$
The position of $u_1$ in ${\bf f}$ is congruent to $0\pmod{8}$. Thus
the initial $xy$ occurs at a position in ${\bf f}$ that is congruent
to $1\pmod{8}$. Hence $x=y=0$ and so
$\Delta(w')=M(4n+1)+2$.
\end{proof}

\begin{claim}
$\rho(16n+7)=\rho(2n+1)+2$.
\end{claim}

\begin{proof}
Since $\rho(4n+2) = \rho(2n+1)+1$, to prove this result we show that
$\rho(16n+7)=\rho(4n+2)+1.$ Let $w'=w_1w_2\cdots w_{4n+2}\in
B_{4n+2}$ such that $\Delta(w')=M(4n+2).$ Then we know that
$$
xy\overline{x}w_1x\overline{yx}w_2\cdots x\overline{yx}w_{4n+2}xy\overline{x}\in B_{16n+11}.
$$
For convenience let us define the following:
\begin{eqnarray*}
z_1&:=&xy\overline{x}w_1x\overline{yx}w_2\cdots w_{4n+1}x\overline{yx},\\
z_2&:=&y\overline{x}w_1x\overline{yx}w_2\cdots x\overline{yx}w_{4n+2},\\
z_3&:=&\overline{x}w_1x\overline{yx}w_2\cdots x\overline{yx}w_{4n+2}x, \mbox{and}\\
z_4&:=&w_1x\overline{yx}w_2\cdots x\overline{yx}w_{4n+2}xy
\end{eqnarray*}
so that
\begin{eqnarray*}
\Delta(z_1) &=& \Delta(4n+1) \leq M(4n+2)+1,\\
\Delta(z_2)&=&M(4n+2)+\Delta(\overline{x}),\\
\Delta(z_3)&=&M(4n+2)+\Delta(\overline{y}),\\
\Delta(z_4)&=&M(4n+2)+\Delta(x).
\end{eqnarray*}
Also, by the maximality of $\Delta(w')$ we must have that $$M(16n+7)\in\{\Delta(z_1),\Delta(z_2),\Delta(z_3),\Delta(z_4)\}.$$ If $x=0$ then it can easily be verified that $\Delta(z_i)\leq\Delta(z_4)$ for $i=1,2,3$. Therefore $M(16n+7)=\Delta(z_4)=M(4n+2)+1.$ Similarly if $x=1$, then $M(16n+7)=\Delta(z_2)=M(4n+2)+1.$ Therefore $\rho(16n+7)=\rho(4n+2)+1=\rho(2n+1)+2.$
\end{proof}

\begin{claim}
$\rho(16n+9)=\rho(2n+1)+2$.
\end{claim}

\begin{proof}
Let $w'=w_1w_2\cdots w_{4n+2}\in B_{4n+2}$ such that $\Delta(w')=M(4n+2).$ Then we know that 
$$
w':=xy\overline{x}w_1x\overline{yx}w_2\cdots x\overline{yx}w_{4n+2}xy\overline{x}\in B_{16n+11}.
$$
For convenience let us define the following:
\begin{eqnarray*}
z_1&:=&xy\overline{x}w_1x\overline{yx}w_2\cdots w_{4n+1}x\overline{yx}w_{4n+2}x,\\
z_2&:=&\overline{x}w_1x\overline{yx}w_2\cdots x\overline{yx}w_{4n+2}xy\overline{x}.
\end{eqnarray*}
Then we have $z_1,z_2\in B_{16n+9}$, $\Delta(z_1)=M(4n+2)+\Delta(x)=M(16n+8)+\Delta(x)$, and $\Delta(z_2)=M(16n+8)+\Delta(\overline{x})$. Therefore, regardless of the value of $x$, we can find a factor $v\in\{z_1,z_2\}$ such that $\Delta(v)=M(16n+8)+1$. Therefore $M(16n+9)\geq M(16n+8)+1$. But we also know that $M(16n+9)\leq M(16n+8)+1$. Therefore $M(16n+9)=M(16n+8)+1=M(2n+1)+2$, and the result follows. 
\end{proof}

\begin{claim}
$\rho(16n+11)=\rho(4n+3)+2$.
\end{claim}

\begin{proof}
Let $w\in B_{16n+11}$ such that $\Delta(w)=M(16n+11)$. Then this factor occurs (in the paperfolding word) at position 1, 2, 3, or 4 of a factor of the form
$$
xy\overline{x}w_1x\overline{yx}w_2\cdots xy\overline{x}w_{4n+3}x\overline{y}.
$$
Regardless of the position where $w$ starts, it is easily verified that $$\Delta(w)\leq\Delta(w_1w_2\cdots w_{4n+3})+2\leq M(4n+3)+2.$$ Thus it suffices to find a $w\in B_{16n+11}$ such that $\Delta(w)=M(4n+3)+2$. 

Let $v=w_1 w_2\cdots w_{4n+3}\in B_{4n+3}$ such that $\Delta(v)=M(4n+3)$. We have either
\begin{enumerate}
\item $v=zu_1\overline{z}u_2\cdots zu_{2n+1}\overline{z}$ or 
\item $v=u_1\overline{z}u_2\cdots zu_{2n+1}zu_{2n+1}\overline{z}u_{2n+2}$,
\end{enumerate}
where $u_i,z\in\{0,1\}$. However, case (i) can always be reduced to case (ii). So let $v$ be of the form corresponding to case (ii) above. Then we know
$$
w':=u_1x\overline{yx}u_1xy\overline{xz}\cdots x\overline{yx}\overline{z}xy\overline{x}u_{2n+2}x\overline{y}
$$
is a factor of the paperfolding word. The position of $u_1$ in ${\bf f}$ is congruent to $0\pmod{8}$. Then $xy=01$ and so $\Delta(w')=M(4n+3)+2$.
\end{proof}

\begin{claim}
$\rho(16n+13)=\rho(2n+1)+2$.
\end{claim}

\begin{proof}
Let $w\in B_{16n+13}$ such that $\Delta(w)=M(16n+13)$. Then $w$ occurs in the paperfolding word as a factor of
\begin{enumerate}
\item $xy\overline{x}zx\overline{yx}u_1xy\overline{xz}x\overline{yx}u_2\cdots x\overline{yx}u_{2n+2}$, or
\item $xy\overline{x}u_1x\overline{yx}\overline{z}xy\overline{x}u_2x\overline{yx}z\cdots u_{2n+2}x\overline{yx}z$.
\end{enumerate}
In both cases we have that the position of $u_1$ in ${\bf f}$ is congruent to $0\pmod{8}$. Thus in case (i) we have $x=y=0$ and in case (ii) we have $xy=01$. In either case it can be verified that $\Delta(w)\leq\Delta(u_1\cdots u_{2n+1})+2\leq M(2n+1)+2$.

Now let $\Delta(u_1u_2\cdots u_{2n+1})=M(2n+1)$. Then
$$w=xy\overline{x}zx\overline{yx}u_1xy\overline{xz}x\overline{yx}u_2\cdots
x\overline{yx} u_{2n+1}xy\overline{xz}x \in B_{16n+3}$$
and $x=y=0$. Hence $\Delta(w)=M(2n+1)+2$ and so $M(16n+3)\geq M(2n+1)+2$. The result follows immediately.
\end{proof}

\begin{claim}
$\rho(16n+15)=\rho(2n+2)+1$.
\end{claim}

\begin{proof}
First note that $\rho(2n+2)+1=\rho(4n+4)+1.$ Let $w':=w_1w_2\cdots w_{4n+4}\in B_{4n+4}$ such that $\Delta(w')=M(4n+4).$ Then
$$
xy\overline{x}w_1x\overline{yx}w_2\cdots x\overline{yx}w_{4n+4}xy\overline{x}\in B_{16n+19},
$$
and we let
$$
z_1:=y\overline{x}w_1x\overline{yx}w_2\cdots x\overline{yx}w_{4n+4}, 
$$
and
$$
z_2:=w_1x\overline{yx}w_2\cdots x\overline{yx}w_{4n+4}xy.
$$
Then $\Delta(z_1)=M(4n+4)+\Delta(\overline{x})$ and $\Delta(z_2)=M(4n+4)+\Delta(x)$ so that either $\Delta(z_1)$ or $\Delta(z_2)$ is $M(4n+4)+1.$ Since $M(16n+15)\leq M(16n+16)+1=M(4n+4)+1$, we have that $\rho(16n+15)=\rho(4n+4)+1=\rho(2n+2)+1.$
\end{proof}

This completes the proof of Theorem~\ref{identities} (and thus of
Theorem~\ref{2-regular}).

\section{Growth of $\rho(n)$}

We now apply Theorem~\ref{identities} to deduce some information
concerning the growth of the function $\rho(n)$.  Unlike the subword
complexity function, which is strictly increasing for any aperiodic
word, the abelian complexity function can fluctuate considerably.  For
instance, in the case of the paperfolding word we have $\rho(2^n)=3$
for $n \geq 1$ (indeed, apart from the initial value $\rho(1)=2$, the
value $3$ is the smallest value taken by the function $\rho$).  On the
other hand, we have $\rho(n) = \lceil \log_2 (n) \rceil + 2$
infinitely often.  To see where these ``large'' values occur, we
define the sequence
\[
A(i) := \min \{n \in \mathbb{N} : \rho(n) = i+1\}.
\]

\begin{prop}\label{growth}
For $i\geq 1$ we have
\[
A(i) =
\begin{cases}
\displaystyle\frac{2^i+1}{3} &\text{if $i$ is odd,}\\
\displaystyle\frac{2^i+2}{3} &\text{if $i$ is even.}
\end{cases}
\]
\end{prop}

\begin{proof}
We begin by defining
\[
B(i) :=
\begin{cases}
\frac{2^i+1}{3} &\text{if $i$ is odd,}\\
\frac{2^i+2}{3} &\text{if $i$ is even.}
\end{cases}
\]
Note that $B(i+1) \leq 2B(i)$ for $i\geq 1$.

We will show that $A(i) = B(i)$.  We begin by showing that $\rho(B(i))
= i+1$.  Suppose first that $i$ is odd, so that $i=2r+1$.  We have
\begin{eqnarray*}
\rho(B(i)) & = & \rho\left(\frac{2^i+1}{3}\right) \\
& = & \rho\left(\frac{2^{2r+1}+1}{3}\right) \\
& = & \rho\left(\frac{2\cdot 4^r+1}{3}\right).
\end{eqnarray*}
Observe that $(2\cdot 4^r+1)/3 \equiv 11 \pmod{16}$ for $r \geq 2$, so
that we may repeatedly apply the identity $\rho(16n+11) =
\rho(4n+3)+2$ of Theorem~\ref{identities} to obtain
\begin{eqnarray*}
\rho(B(i)) & = &  \rho\left(\frac{2\cdot 4^r+1}{3}\right)\\
& = &  \rho\left(\frac{2\cdot 4^{r-1}+1}{3}\right) + 2\\
& = &  \rho\left(\frac{2\cdot 4^{r-2}+1}{3}\right) + 2 + 2\\
& \vdots &\\
& = & \rho(3) + (r-1)2 = 4 + 2r - 2 = 2r+2 = i+1.
\end{eqnarray*}

Now suppose that $i$ is even.  Observe that when $i$ is even, we have
$(2^i+2)/3 \equiv 2 \pmod 4$, so that we may apply the identity
$\rho(4n+2) = \rho(2n+1)+1$ to obtain $$\rho(B(i)) = \rho((2^i+2)/3) =
\rho((2^{i-1} + 1)/3) + 1 = \rho(B(i-1)) + 1 = i +1,$$ where in the
last step we have used the formula proved above for the odd case.

We have established that $A(i) \leq B(i)$.  We show that $A(i) \geq
B(i)$ by induction on $i$.  We have $\rho(1) = 2$ and $\rho(2) = 3$,
so the result holds for $i=1,2$.  Suppose that for $m < B(i)$ we have
$\rho(m) \leq i$.  Now consider $m \in \{B(i)+1, B(i)+2, \ldots,
B(i+1)-1\}$.  We wish to show that $\rho(m) \leq i+1$.  There are
various cases to consider depending on the congruence class of $m$
modulo $16$.

For instance, if $m=4n+2$, then by Theorem~\ref{identities}, we have
$\rho(m) = \rho(4n+2) = \rho(2n+1)+1$, and since $2n+1 < B(i)$, by the
induction hypothesis we have $\rho(2n+1) \leq i$.  Thus $\rho(m) \leq
i+1$, as required.

Suppose $m=16n+3$.  Then $\rho(m) = \rho(16n+3) = \rho(2n+1) + 2$, and
since $2n+1 < B(i-1)$, by the induction hypothesis we have $\rho(2n+1)
\leq i-1$.  Therefore $\rho(m) \leq i-1+2 = i+1$, as required.

Suppose $m=16n+11$.  Then $\rho(m) = \rho(16n+11) = \rho(4n+3) + 2$.
Since $4n+3 < B(i-1)$, we have $\rho(4n+3) \leq i-1$, so $\rho(m) \leq
i-1+2 = i+1$, as required.

For the other congruence classes modulo 16, the proofs are analogous.
This completes the inductive proof that $A(i) \geq B(i)$.  Putting the
two inequalities together, we get $A(i) = B(i)$, as claimed.
\end{proof}

By taking logarithms in Proposition~\ref{growth}, we get an upper
bound of $\rho(n) \leq \lceil \log_2 (n) \rceil + 2$.

\section{Conclusion}

The present work leaves open some natural problems/questions, including:

\begin{enumerate}
\item Determine the abelian complexity function for all paperfolding
  words.
\item Is the abelian complexity function of a $k$-automatic sequence
  always $k$-regular?
\end{enumerate}

\end{document}